\documentclass[12pt,a4paper]{article}
\usepackage{amsmath}
\usepackage{amsfonts}
\usepackage{amssymb}
\usepackage{amsthm}
\usepackage{mathtools}
\usepackage{mathabx}
\usepackage{xcolor}
\usepackage{hyperref}
\usepackage[latin1]{inputenc} 
\usepackage[french,english]{babel}  
\usepackage{paralist}
\usepackage[numbers]{natbib}
\usepackage{bbm}
\usepackage{pifont}

\newtheorem{proposition}{Proposition}
\newtheorem{definition}{Definition}
\newtheorem{theorem}{Theorem}
\newtheorem{remark}{Remark}
\newtheorem{lemma}{Lemma}

\newcommand{\N}{\mathbb{N}}

\newcommand{\cf}{\mathcal{F}}

\newcommand{\F}{\mathbb{F}}

\def\esssup{\text{ess sup}} 
\DeclareMathOperator*{\essinf}{ess\,inf}

\makeatletter
\renewenvironment{proof}[1][\proofname]{%
  \par\pushQED{\qed}\normalfont%
  \topsep6\p@\@plus6\p@\relax
  \trivlist\item[\hskip\labelsep\bfseries#1\@addpunct{.}]%
  \ignorespaces
}{%
  \popQED\endtrivlist\@endpefalse
}

\author{ 
 Miryana Grigorova$^1$\footnote{Corresponding Author. Department of Statistics, University of Warwick, E-mail:miryana.grigorova@warwick.ac.uk} \\
 \and
 Marie-Claire Quenez$^2$ \footnote{LPSM, University Paris-Cit\'e} 
 \and
 Peng Yuan$^3$ \footnote{Department of Statistics, University of Warwick} 
}

\title{Non-linear non-zero-sum Dynkin games with Bermudan strategies}
\date{%
    $^{1,3}$University of Warwick\\%
    $^2$Universit\'e Paris Cit\'e\\[2ex]%
    \today
}
\begin{document}

\maketitle

\textit{Abstract:} In this paper, we study a non-zero-sum game with two players, where each of the players plays what we call Bermudan strategies and optimizes a general non-linear assessment functional of the pay-off. By using a recursive construction, we show that the game has a Nash equilibrium point.

\section{Introduction} Game problems with linear evaluations between a finite number of players are by now classical problems in stochastic control and optimal stopping (cf., e.g., \cite{Alario-Nazaret}, \cite{Bismut-1}, \cite{Cvitanic-Karatzas}, \cite{Dynkin}, \cite{Hamadene-1}, \cite{Hamadene-Zhang-1}, \cite{Hamadene-Hassani-1}, \cite{Hamadene-Hassani-2}, \cite{Hamadene-Hassani-3}, \cite{Hamadene-Lepeltier}, \cite{Kobylanski-Quenez-Rouy-Mironescu}, \cite{Kobylanski-Quenez-Roger} and \cite{Lepeltier-Maingueneau}) with various applications, in particular in economics and finance (cf., e.g., \cite{Hamadene-1}, \cite{Hamadene-Zhang-1}, \cite{Kifer-1} and \cite{Kobylanski-Quenez-Rouy-Mironescu}). In the recent years game problems with \emph{non-linear} evaluation functionals have attracted considerable interest: cf. \cite{Bayraktar-Yao-1} for the case of non-linear functionals of the form of worst case expectations over a set of possibly singular measures;  \cite{Dumitrescu-2}, \cite{Dumitrescu-3}, \cite{Dumitrescu-1} and \cite{Grigorova-Imkeller-Offen-Ouknine-Quenez} for the case of non-linear functionals induced by backward stochastic differential equations (BSDEs). Most of the works dealing with \emph{non-linear} games have focused on the \emph{zero-sum} case (cf., e.g., \cite{Bayraktar-Yao-1}, \cite{Dumitrescu-2}, \cite{Dumitrescu-3}, \cite{Dumitrescu-1} and \cite{Grigorova-Imkeller-Offen-Ouknine-Quenez}). \emph{Non-zero-sum} games are notoriously more intricate than their zero-sum counterparts even in the case of linear evaluations (cf., e.g., \cite{Hamadene-Zhang-1}, \cite{Hamadene-Hassani-2}, \cite{Hamadene-Hassani-3}, \cite{Laraki-Solan}, \cite{Morimoto-1} and \cite{Ohtsubo}). Non-zero-sum games with non-linear functionals have been considered in \cite{Grigorova-Quenez-1} in the discrete-time framework and with non-linear functionals induced by Backward SDEs with Lipschitz driver, in \cite{Kentia-Kuhn} in the continuous time framework and with non-linear functionals of the form of expected exponential utilities.\\
[0.2cm]
In the current paper, we address the question of existence of a Nash equilibrium point in a framework with general non-linear evaluations and with a set of stopping strategies which is in between the discrete time and the continuous time stopping strategies. The results of \cite{Grigorova-Quenez-1} can be seen as a particular case of the current paper.\\
[0.2cm]
The paper is organised as follows: In Section \ref{Section2}, we introduce the framework, including the set of optimal stopping strategies of the agents (namely the Bermudan strategies), the pay-off as well as the properties on the risk functionals $\rho^{1}$ and $\rho^{2}$ of agent 1 and agent 2. In Section \ref{Section3}, we present our main results and show that the non-linear non-zero-sum game with Bermudan strategies has a Nash equilibrium point.

\section{The framework} \label{Section2}
Let $T>0$ be a \textbf{fixed finite} terminal horizon.\\
[0.2cm]
Let $(\Omega,  \cf, P)$ be a (complete) probability space equipped with a right-continuous complete filtration $\F = \{\mathcal{F}_t \colon t\in[0,T]\}$. \\
[0.2cm]
In the sequel, equalities and inequalities between random variables are to be understood in the $P$-almost sure sense. Equalities between measurable sets are to be understood in the $P$-almost sure sense.\\
[0.2cm]
Let $\N$ be the set of natural numbers, including $0$.  Let $\N^*$ be the set of natural numbers, excluding $0$.\\
[0.2cm] 
We first define the so-called Bermudan stopping strategies (introduced in \cite{Grigorova-1}).\\
[0.2cm]
Let ($\theta_k)_{k\in\N}$ be a sequence of stopping times satisfying the following properties:
\begin{itemize}
\item[(a)] The sequence $(\theta_k)_{k\in\N}$  is non-decreasing, i.e. 
for all $k\in\N$, $\theta_k\leq \theta_{k+1}$, a.s. 
\item[(b)] $\lim_{k\to\infty}\uparrow \theta_k=T$ a.s. 
\end{itemize}
Moreover, we set $\theta_0=0$. \\
[0.2cm]
We note that the family of $\sigma$-algebras 
$({\cal F}_{\theta_k})_{k\in\N}$ is non- decreasing (as the sequence $(\theta_{k})$ is non-decreasing).
We denote by ${\Theta}$ the set of stopping times $\tau$ of the form 
\begin{equation}\label{form}
\tau= \sum_{k=0}^{+\infty} \theta_k {\bf 1}_{A_k}+T{\bf 1}_{\bar A},
\end{equation}
where $\{(A_{k})^{+\infty}_{k=0}, \bar{A}\}$ form a partition of $\Omega$ such that, for each $k\in\N$,  $A_{k} \in \mathcal{F}_{\theta_{k}}$, and $\bar{A} \in \mathcal{F}_{T}$.\\
[0.2cm]
The set ${\Theta}$ can also  be described as the set of stopping times $\tau$ such that for almost all $\omega \in \Omega$, either $\tau(\omega) = T$ or $\tau(\omega) = \theta_{k}(\omega)$, for some $k = k(\omega) \in \mathbbm{N}$.\\
[0.2cm]
Note that the set $\Theta$ is closed under concatenation, that is, for each $\tau$ $\in$ $\Theta$ and each 
$A \in {\cal F}_{\tau}$, the stopping time $\tau {\bf 1}_{A} + T {\bf 1}_{A^c}$ $\in$ $\Theta$. More generally, for each $\tau$ $\in$ $\Theta$, $\tau'$ $\in$ $\Theta$ and each 
$A \in {\cal F}_{\tau\wedge \tau'}$, the stopping time $\tau {\bf 1}_{A} + \tau' {\bf 1}_{A^c}$ is in  $\Theta$.  The set $\Theta$ is also closed under pairwise minimization (that is, for  each $\tau\in\Theta$ and $\tau'\in\Theta$, we have $\tau\wedge \tau'\in\Theta$) and under pairwise maximization (that is, for  each $\tau\in\Theta$ and $\tau'\in\Theta$, we have $\tau\vee \tau'\in\Theta$). Moreover, the set $\Theta$ is closed under monotone limit, that is, for each non-decreasing (resp. non-increasing) sequence of stopping times $(\tau_n)_{n \in {\mathbb N}} \in \Theta^{\mathbbm{N}}$, we have $\lim_{n \rightarrow + \infty} \tau_n$  $\in$ $\Theta$.\\
[0.2cm]
We note also that all stopping times in $\Theta$ are bounded from above by $T$. 
\begin{remark}\label{Rk_canonical} 
We have the following \emph{canonical} writing of the sets in \eqref{form}: 
\begin{align*}
A_0&=\{\tau=\theta_0\};\\
A_{n+1}&=\{\tau=\theta_{n+1},\theta_{n+1}<T\}\backslash (A_n\cup...\cup A_0); \text { for all } n\in\N^*\\
\bar A&=(\cup_{k=0}^{+\infty} A_k)^c  
\end{align*}  
From this writing, we have: if $\omega\in A_{k+1}\cap \{\theta_k<T\}$, then $\omega\notin \{\tau=\theta_k\}.$ 
\end{remark}
\noindent
For each $\tau \in \Theta$, we denote by $\Theta_{\tau}$ the set of stopping times $\nu \in 
\Theta$ such that $\nu \geq \tau$ a.s.\, The set $\Theta_{\tau}$ satisfies the same properties as 
the set $\Theta$.
We will refer to the set $\Theta$ as the set of \textbf{Bermudan stopping strategies},
and to the set $\Theta_\tau$ as the set of Bermudan stopping strategies, greater than or equal to $\tau$ (or the set of Bermudan stopping strategies from time $\tau$ perspective). For simplicity, the set $\Theta_{\theta_{k}}$ will be denoted by $\Theta_{k}$.

\begin{definition}\label{def.admi}
We say that a family 
$\phi=(\phi(\tau), \, \tau \in \Theta)$  is \emph{admissible} if it satisfies the following conditions 
\par
1. \quad for all
$\tau \in \Theta$, $\phi(\tau)$ is a real valued random variable, which is  $\mathcal{F}_\tau$-measurable. \par
 2. \quad  for all
$\tau,\tau'\in \Theta$, $\phi(\tau)=\phi(\tau')$ a.s.  on
$\{\tau=\tau'\}$.

Moreover, for $p \in [1, +\infty]$ fixed, we say that an admissible family $\phi$ is $p$-integrable, if for all $\tau \in \Theta$, $\phi(\tau)$ is in $L^{p}$.
\end{definition}
\noindent
Let $\phi=(\phi(\tau), \, \tau \in \Theta)$ be an admissible family. For a stopping time $\tau$  of the form \eqref{form}, we have 
\begin{equation}\label{formula}
\phi(\tau)= \sum_{k=0}^{+\infty} \phi(\theta_k ){\bf 1}_{A_k}+\phi(T){\bf 1}_{\bar A} \quad {\rm a.s.}
\end{equation}
Given two admissible families $\phi=(\phi(\tau), \, \tau \in \Theta)$ and $\phi'=(\phi'(\tau), \, \tau \in \Theta)$, we 
say that $\phi$ is {\em equal to} $\phi'$ and write $\phi = \phi'$ if, for all $\tau \in \Theta$, 
$\phi(\tau) =\phi'(\tau)$ a.s. We 
say that $\phi$ {\em dominates} $\phi'$ and write $\phi \geq \phi'$
if, for all $\tau \in \Theta$, 
$\phi(\tau) \geq \phi'(\tau)$ a.s.\\
[0.2cm] 
Let $p \in [1, +\infty]$. We introduce the following properties  on the non-linear operators $\rho_{S, \tau}[\cdot]$, which will appear in the sequel. \\
[0.2cm] 
For
$S\in\Theta$, $S'\in\Theta$, $\tau\in\Theta$, for $\eta$, $\eta_{1}$ and $\eta_2$ in $L^p(\cf_\tau)$, for $\xi=(\xi(\tau))$ an admissible p-integrable family:  
\begin{compactenum}[(i)]
\item[(i)] $\rho_{S,\tau}: L^p(\cf_\tau)\longrightarrow L^p(\cf_S)$ 
\item[(ii)] \emph{(admissibility)} $\rho_{S,\tau}[\eta]=\rho_{S',\tau}[\eta]$ a.s.  on $\{S=S'\}$. 
\item[(iii)] \emph{(knowledge preservation)}
$\rho_{\tau,S}[\eta]=\eta, $
for all $\eta\in L^p(\cf_S)$, all $\tau\in\Theta_S.$
\item[(iv)] \emph{(monotonicity)}  $\;\rho_{S,\tau}[\eta_1]\leq \rho_{S,\tau}[\eta_2]$ a.s., if  $\eta_1\leq \eta_2$ a.s. 
\item[(v)] \emph{(consistency)}  $\;\rho_{S,\theta}[\rho_{\theta,\tau}[\eta]]= \rho_{S,\tau}[\eta]$, for all $S, \theta, \tau$ in $\Theta$ such that $S\leq \theta\leq \tau$ a.s.  
\item [(vi)] \emph{("generalized zero-one law") }$\; I_A\rho_{S,\tau}[\xi(\tau)] =I_A\rho_{S,\tau'}[\xi(\tau')],$ for all $A\in\cf_S$, $\tau\in \Theta_{S}$, $\tau'\in\Theta_{S}$ such that $\tau=\tau'$ on $A$.  
\item[(vii)] \emph{(monotone Fatou property with respect to terminal condition)}\\
$\rho_{S, \tau}[\eta] \leq \liminf_{n \to +\infty} \rho_{S, \tau}[\eta_{n}]$, for $(\eta_{n}), \eta$ such that $(\eta_{n})$ is non-decreasing, $\eta_{n} \in L^{p}(\cf_{\tau})$, $\sup_{n}\eta_{n} \in L^{p}$, and $\lim_{n \to +\infty} \uparrow \eta_{n} = \eta$ a.s.
\item[(viii)] (left-upper-semicontinuity (LUSC) along Bermudan stopping times with respect to the terminal condition and the terminal time), that is,  
$$\limsup_{n \to +\infty}\rho_{S, \tau_{n}}[\phi(\tau_{n})] \leq \rho_{S, \nu}[\limsup_{n \to +\infty} \phi(\tau_{n})],$$
$\text{for each non-decreasing sequence }(\tau_{n}) \in \Theta_{S}^{\N} \text{ such that }\lim_{n \to +\infty} \uparrow \tau_{n} = \nu \text{ a.s.},$ and for each $p$-integrable admissible family $\phi$ such that $\sup_{n\in\N} |\phi(\tau_n)|\in L^p.$
\item[(ix)] $\limsup_{n \to +\infty}\rho_{\theta_{n}, T}[\eta] \leq \rho_{T, T}[\eta], \text { for all } \eta\in L^p(\cf_T).$\\
\end{compactenum}
These assumptions on $\rho$ ensure that the one-agent's non-linear optimal stopping problem admits a solution and that the first hitting time (when the value family ``hits'' the pay-off family) is optimal (cf. \cite{Grigorova-1} for more details).

\section{The game problem}\label{Section3}
We consider two agents, agent 1 and agent 2, whose pay-offs are defined via four admissible families $X^{1} = (X^{1}(\tau))_{\tau \in \Theta}$, $X^{2} = (X^{2}(\tau))_{\tau \in \Theta}$, $Y^{1} = (Y^{1}(\tau))_{\tau \in \Theta}$ and $Y^{2} = (Y^{2}(\tau))_{\tau \in \Theta}$. We assume that $X^{1}, X^{2}, Y^{1}$ and $Y^{2}$ are $p$-integrable families such that \\
\begin{compactenum}[(i)]
\item[(A1)] $X^{1} \leq Y^{1}, \; X^{2} \leq Y^{2}$ (that is, for each $\tau \in \Theta$, $X^{1}(\tau) \leq Y^{1}(\tau)$, and $X^{2}(\tau) \leq Y^{2}(\tau)$).\\
\item[(A2)] $X^{1}(T) = Y^{1}(T), \; X^{2}(T) = Y^{2}(T)$.\\
\item[(A3)] $\esssup_{\tau \in \Theta}X^{1}(\tau) \in L^{p}$, $\esssup_{\tau \in \Theta}X^{2}(\tau) \in L^{p}$,\\
[0.2cm]
$\esssup_{\tau \in \Theta}Y^{1}(\tau) \in L^{p}$ and $\esssup_{\tau \in \Theta}Y^{2}(\tau) \in L^{p}$.\\
\item[(A4)] $\limsup_{k \to +\infty}X^{1}(\theta_{k}) \leq X^{1}(T)$, $\;\;$   $\limsup_{k \to +\infty}X^{2}(\theta_{k}) \leq X^{2}(T)$.\\

\end{compactenum}
The set of stopping strategies of each agent at time $0$ is the set $\Theta$ of Bermudan stopping times. If the first agent plays $\tau_{1} \in \Theta$ and the second agent plays $\tau_{2} \in \Theta$, the pay-off of agent 1 (resp. agent 2) at time $\tau_{1} \wedge \tau_{2}$ is given by: 
$$I^{1}(\tau_{1}, \tau_{2}) \coloneqq X^{1}(\tau_{1})\mathbbm{1}_{\{\tau_{1} \leq \tau_{2}\}} + Y^{1}(\tau_{2})\mathbbm{1}_{\{\tau_{2} < \tau_{1}\}}$$ 
$$(\text{resp. } I^{2}(\tau_{1}, \tau_{2}) \coloneqq X^{2}(\tau_{2})\mathbbm{1}_{\{\tau_{2} < \tau_{1}\}} + Y^{2}(\tau_{1})\mathbbm{1}_{\{\tau_{1} \leq \tau_{2}\}}),$$
where we have adopted the convention: when $\tau_{1} = \tau_{2}$, it is the first agent who is responsible for stopping the game. The agents evaluate their respective pay-offs via possibly different evaluation functionals. Let $\rho^{1} = (\rho_{S, \tau}[\cdot])$ be the family of evaluation operators of agent 1, and let $\rho^{2} = (\rho_{S, \tau}[\cdot])$ be the family of evaluation operators of agents 2. If agent 1 plays $\tau_{1} \in \Theta$, and agent 2 plays $\tau_{2} \in \Theta$, then the assessment (or evaluation) of agent 1 (resp. agent 2) at time $0$ of his/her pay-off is given by:
$$J_{1}(\tau_{1}, \tau_{2}) \coloneqq \rho^{1}_{0, \tau_{1} \wedge \tau_{2}}[X^{1}(\tau_{1})\mathbbm{1}_{\{\tau_{1} \leq \tau_{2}\}} + Y^{1}(\tau_{2})\mathbbm{1}_{\{\tau_{2} < \tau_{1}\}}].$$
$$(\text{resp. } J_{2}(\tau_{1}, \tau_{2}) \coloneqq \rho^{2}_{0, \tau_{1} \wedge \tau_{2}}[X^{2}(\tau_{2})\mathbbm{1}_{\{\tau_{2} < \tau_{1}\}} + Y^{2}(\tau_{1})\mathbbm{1}_{\{\tau_{1} \leq \tau_{2}\}}]).$$
We assume that both $\rho^{1}$ and $\rho^{2}$ satisfy the properties (i) - (ix). We will investigate the problem of existence of a Nash equilibrium strategy $(\tau_{1}^{*}, \tau_{2}^{*})$.

\begin{definition}
A pair of Bermudan stopping times $(\tau_{1}^{*}, \tau_{2}^{*}) \in \Theta \times \Theta$ is called a Nash equilibrium strategy (or a Nash equilibrium point) for the above non-zero-sum non-linear Bermudan Dynkin game if: $J_{1}(\tau_{1}^{*}, \tau_{2}^{*}) \geq J_{1}(\tau_{1}, \tau_{2}^{*})$, for any $\tau_{1} \in \Theta$, and $J_{2}(\tau_{1}^{*}, \tau_{2}^{*}) \geq J_{2}(\tau_{1}^{*}, \tau_{2})$, for any $\tau_{2} \in \Theta$.

\end{definition}
\noindent
In other words, any unilateral deviation from the strategy $(\tau_{1}^{*}, \tau_{2}^{*})$ by one of the agent (the strategy of the other remaining fixed) does not render the deviating agent better off.

\begin{theorem}\label{TEXtheorem1}
Under assumptions (i) - (ix) on $\rho^{1}$ and $\rho^{2}$, there exists a Nash equilibrium point $(\tau_{1}^{*}, \tau_{2}^{*})$ for the game described above.

\end{theorem}
\noindent
We will construct a sequence $(\tau_{2n+1}, \tau_{2n})_{n \in \mathbbm{N}}$ (by induction), for which we will show that it converges to a Nash equilibrium point.\\
[0.2cm]
We set $\tau_{1} \coloneqq T$ and $\tau_{2} \coloneqq T$. We suppose that $\tau_{2n-1} \in \Theta$ and $\tau_{2n} \in \Theta$ have been defined. We set, for each $k \in \mathbbm{N}$, 
\begin{equation} \label{TEXeQuationN2}
\xi^{2n+1}(\theta_{k}) \coloneqq X^{1}(\theta_{k})\mathbbm{1}_{\{\theta_{k} < \tau_{2n}\}} + Y^{1}(\tau_{2n})\mathbbm{1}_{\{\tau_{2n} \leq \theta_{k}\}}.
\end{equation}
Moreover, $\xi^{2n+1}(T) \coloneqq Y^{1}(\tau_{2n})$. This definition is ``consistent'' with the above, as by \eqref{TEXeQuationN2}, $\mathbbm{1}_{\{\theta_{k}=T\}}\xi^{2n+1}(\theta_{k}) = \mathbbm{1}_{\{\theta_{k}=T\}}Y^{1}(\tau_{2n})$.\\
[0.3cm]
For $\tau \in \Theta$ of the form $\tau = \sum_{k\in\mathbbm{N}}\theta_{k}\mathbbm{1}_{A_{k}} + T\mathbbm{1}_{\bar{A}}$, where $((A_{k}), \bar{A})$ is a partition, $A_{k}$ is $\mathcal{F}_{\theta_{k}}$-measurable for each $k \in \mathbbm{N}$, and $\bar{A}$ is $\mathcal{F}_{T}$-measurable,
\begin{equation}\label{TEXeQuationN22}
\xi^{2n+1}(\tau) \coloneqq \sum_{k\in\mathbbm{N}}\xi^{2n+1}(\theta_{k})\mathbbm{1}_{A_{k}} + \xi^{2n+1}(T)\mathbbm{1}_{\bar{A}}.
\end{equation}
We note that $\xi^{2n+1}(\theta_{k})$ is the pay-off at $\theta_{k} \wedge \tau_{2n}$ of agent 1 (up to the equality $\{\theta_{k} = \tau_{2n}\}$) if agent 1 plays $\theta_{k}$ and agent 2 plays $\tau_{2n}$.\\
[0.2cm]
We also note that: 
$$\xi^{2n+1}(\tau) = X^{1}(\tau)\mathbbm{1}_{\{\tau<\tau_{2n}\}} + Y^{1}(\tau_{2n})\mathbbm{1}_{\{\tau_{2n}\leq\tau\}}.$$
Thus, $\xi^{2n+1}(\tau)$ is the pay-off at $\tau \wedge \tau_{2n}$ of agent 1 (up to the equality $\{\tau = \tau_{2n}\}$) if agent 1 plays $\tau$ and agent 2 plays $\tau_{2n}$.\\
[0.2cm]
For each $S \in \Theta$, we define
\begin{equation}\label{TEXeQuationN55}
\begin{aligned}
&V^{2n+1}(S) \coloneqq \esssup_{\tau \in \Theta_{S}} \rho^{1}_{S, \tau \wedge \tau_{2n}}[\xi^{2n+1}(\tau)]\\
&\tilde{\tau}_{2n+1} \coloneqq \essinf \tilde{\mathcal{A}}^{1}, \text{ where } \tilde{\mathcal{A}}^{1} \coloneqq \{\tau \in \Theta: V^{2n+1}(\tau) = \xi^{2n+1}(\tau)\}\\
&\tau_{2n+1} \coloneqq (\tilde{\tau}_{2n+1} \wedge \tau_{2n-1})\mathbbm{1}_{\{\tilde{\tau}_{2n+1} \wedge \tau_{2n-1} < \tau_{2n}\}} + \tau_{2n-1}\mathbbm{1}_{\{\tilde{\tau}_{2n+1} \wedge \tau_{2n-1} \geq \tau_{2n}\}}.
\end{aligned}
\end{equation}
Assuming that $\limsup_{k \to +\infty}X^{1}(\theta_{k}) \leq X^{1}(T)$ (from (A4)) ensures that\\ 
$\limsup_{k \to +\infty}\xi^{2n+1}(\theta_{k}) \leq \xi^{2n+1}(T)$. This is a technical condition on the pay-off which we use to apply Theorem 2.3 in \cite{Grigorova-1}.\\
[0.3cm]
We recall that under the assumptions of Theorem 2.3 in \cite{Grigorova-1}, the Bermudan stopping time $\tilde{\tau}_{2n+1}$ is optimal for the optimal stopping problem with value $V^{2n+1}(0)$, that is 
\begin{equation}\label{TEXeQuationN6666}
V^{2n+1}(0) = \rho^{1}_{0, \tilde{\tau}_{2n+1} \wedge \tau_{2n}}[\xi(\tilde{\tau}_{2n+1})] = \sup_{\tau \in \Theta}\rho^{1}_{0, \tau \wedge \tau_{2n}}[\xi^{2n+1}(\tau)].
\end{equation}
We also recall that $V^{2n+1}(T) = \xi^{2n+1}(T)$, under the assumption of knowledge preservation on $\rho$.

\begin{remark}\label{TEXremark2}
i) It is not difficult to show, by induction, that for each $n \in \mathbbm{N}$, $(\xi^{2n+1}(\tau))_{\tau\in\Theta}$ is an admissible $L^{p}$-integrable family, and $\tau_{2n+1}$ is a Bermudan stopping time (for the latter property, we use that $\Theta$ has the property of stability by concatenation of two Bermudan stopping times).\\
[0.2cm]
ii) For each $n \in \mathbbm{N}$, for each $\tau \in \Theta$, $\xi^{2n+1}(\tau) = \xi^{2n+1}(\tau \wedge \tau_{2n})$.\\
[0.2cm]
Indeed, we have: 
\begin{equation}\label{TEXeQuationN4}
\begin{aligned} 
\xi^{2n+1}(\theta_{k}) &= X^{1}(\theta_{k})\mathbbm{1}_{\{\theta_{k}<\tau_{2n}\}} + Y^{1}(\tau_{2n})\mathbbm{1}_{\{\tau_{2n} \leq \theta_{k}\}}\\
&=  X^{1}(\theta_{k}\wedge\tau_{2n})\mathbbm{1}_{\{\theta_{k}\wedge\tau_{2n}<\tau_{2n}\}} + Y^{1}(\tau_{2n})\mathbbm{1}_{\{\tau_{2n} \leq \theta_{k}\wedge\tau_{2n}\}} = \xi^{2n+1}(\theta_{k}\wedge\tau_{2n}).
\end{aligned}
\end{equation}
Now, let $\tau \in \Theta$ be of the form $\tau = \sum_{k\in\mathbbm{N}}\theta_{k}\mathbbm{1}_{A_{k}} + T\mathbbm{1}_{\bar{A}}$. By definition of $\xi^{2n+1}(\tau)$, of $\xi^{2n+1}(T)$ and by Eq. \eqref{TEXeQuationN4}, we have:
\begin{align*}
\xi^{2n+1}(\tau) &= \sum_{k\in\mathbbm{N}}\xi^{2n+1}(\theta_{k})\mathbbm{1}_{A_{k}} + \xi^{2n+1}(T)\mathbbm{1}_{\bar{A}}\\ 
&= \sum_{k\in\mathbbm{N}}\xi^{2n+1}(\theta_{k}\wedge\tau_{2n})\mathbbm{1}_{A_{k}} + Y^{1}(\tau_{2n})\mathbbm{1}_{\bar{A}} = \xi^{2n+1}(\tau\wedge\tau_{2n}).
\end{align*}

\end{remark}

\begin{proposition}\label{TEXProposition1}
i) $\xi^{2n+1}(\tau)\mathbbm{1}_{\{\tau_{2n}  \leq \theta_{k}\}} = Y^{1}(\tau_{2n})\mathbbm{1}_{\{\tau_{2n} \leq \theta_{k}\}}$.\\
[0.2cm] 
ii) $V^{2n+1}(\theta_{k})\mathbbm{1}_{\{\tau_{2n} \leq \theta_{k}\}} = Y^{1}(\tau_{2n})\mathbbm{1}_{\{\tau_{2n} \leq \theta_{k}\}}$.\\
[0.2cm]
iii) $V^{2n+1}(\tau)\mathbbm{1}_{\{\tau_{2n} \leq \tau\}} = Y^{1}(\tau_{2n})\mathbbm{1}_{\{\tau_{2n} \leq \tau\}}$.\\
[0.2cm]
iv) For each $n \in \mathbbm{N}$, $\tilde{\tau}_{2n+1} = \essinf\{\tau \in \Theta: V^{2n+1}(\tau) = X^{1}(\tau)\} \wedge \tau_{2n}$. In particular, $\tilde{\tau}_{2n+1} \leq \tau_{2n}$.

\end{proposition}

\begin{proof}
i) On the set $\{\tau = T\}$, we have $\xi^{2n+1}(\tau) = \xi^{2n+1}(T) = Y^{1}(\tau_{2n})$. On the set $\{\tau = \theta_{k} < T\}$, by the second statement in Remark \ref{TEXremark2}, we have $\xi^{2n+1}(\tau) = \xi^{2n+1}(\theta_{k}) = \xi^{2n+1}(\theta_{k} \wedge \tau_{2n})$. Hence, on the set $\{\tau = \theta_{k} < T\} \cap \{\tau_{2n} \leq \theta_{k}\}$, we have $\xi^{2n+1}(\tau) = \xi^{2n+1}(\tau_{2n}) = Y^{1}(\tau_{2n})$, which proves the desired property.\\
[0.3cm]
ii) We have:
\begin{align*}
\mathbbm{1}_{\{\tau_{2n} \leq \theta_{k}\}}V^{2n+1}(\theta_{k}) &= \mathbbm{1}_{\{\tau_{2n} \leq \theta_{k}\}} \esssup_{\tau \in \Theta_{k}}\rho_{\theta_{k}, \tau\wedge\tau_{2n}}[\xi^{2n+1}(\tau\wedge\tau_{2n})]\\
&=\esssup_{\tau \in \Theta_{k}}\mathbbm{1}_{\{\tau_{2n} \leq \theta_{k}\}}\rho_{\theta_{k}, \tau\wedge\tau_{2n}}[\xi^{2n+1}(\tau\wedge\tau_{2n})]\\
&=\esssup_{\tau \in \Theta_{k}}\mathbbm{1}_{\{\tau_{2n} \leq \theta_{k}\}}\rho_{\theta_{k}, \tau\wedge\tau_{2n}\wedge\theta_{k}}[\xi^{2n+1}(\tau\wedge\tau_{2n}\wedge\theta_{k})],
\end{align*}
where we have used the ``genrealized zero-one law'' to obtain the last equality.\\
[0.2cm]
For any $\tau \in \Theta_{k}$, $\tau\wedge\tau_{2n}\wedge\theta_{k} = \tau_{2n}\wedge\theta_{k} \leq \theta_{k}$. Hence, 
\begin{align*}
\mathbbm{1}_{\{\tau_{2n} \leq \theta_{k}\}}\rho_{\theta_{k}, \tau\wedge\tau_{2n}\wedge\theta_{k}}[\xi^{2n+1}(\tau\wedge\tau_{2n}\wedge\theta_{k})] &= \mathbbm{1}_{\{\tau_{2n} \leq \theta_{k}\}}\rho_{\theta_{k}, \tau_{2n}\wedge\theta_{k}}[\xi^{2n+1}(\tau_{2n}\wedge\theta_{k})]\\
&= \mathbbm{1}_{\{\tau_{2n} \leq \theta_{k}\}}\xi^{2n+1}(\tau_{2n}\wedge\theta_{k}),
\end{align*}
where we have used the knowledge-preserving property of $\rho$ to obtain the last equality.\\
[0.2cm]
Finally, we get
$$\mathbbm{1}_{\{\tau_{2n} \leq \theta_{k}\}}V^{2n+1}(\theta_{k}) = \mathbbm{1}_{\{\tau_{2n} \leq \theta_{k}\}}\xi^{2n+1}(\tau_{2n}) = \mathbbm{1}_{\{\tau_{2n} \leq \theta_{k}\}}Y^{1}(\tau_{2n}).$$
iii) Let $\tau \in \Theta$ be the form $\tau = \sum_{k\in\mathbbm{N}}\theta_{k}\mathbbm{1}_{A_{k}} + T\mathbbm{1}_{\bar{A}}$. Then, by admissibility, we have 
$$V^{2n+1}(\tau) = \sum_{k\in\mathbbm{N}}V^{2n+1}(\theta_{k})\mathbbm{1}_{A_{k}} + V^{2n+1}(T)\mathbbm{1}_{\bar{A}}.$$
Hence, 
\begin{align*}
V^{2n+1}(\tau)\mathbbm{1}_{\{\tau_{2n} \leq \tau\}} &= \sum_{k\in\mathbbm{N}}V^{2n+1}(\theta_{k})\mathbbm{1}_{A_{k} \cap \{\tau_{2n} \leq \tau\}} + V^{2n+1}(T)\mathbbm{1}_{\bar{A} \cap \{\tau_{2n} \leq \tau\}}\\
&= \sum_{k\in\mathbbm{N}}V^{2n+1}(\theta_{k})\mathbbm{1}_{A_{k} \cap \{\tau_{2n} \leq \theta_{k}\}} + V^{2n+1}(T)\mathbbm{1}_{\bar{A} \cap \{\tau_{2n} \leq T\}}\\
&= \sum_{k\in\mathbbm{N}}Y^{1}(\tau_{2n})\mathbbm{1}_{A_{k} \cap \{\tau_{2n} \leq \theta_{k}\}} + \xi^{2n+1}(T)\mathbbm{1}_{\bar{A} \cap \{\tau_{2n} \leq T\}},
\end{align*}
where we have used the previous property (ii) to obtain the last equality. Hence, we get 
\begin{align*}
V^{2n+1}(\tau)\mathbbm{1}_{\{\tau_{2n} \leq \tau\}} &= \sum_{k\in\mathbbm{N}}Y^{1}(\tau_{2n})\mathbbm{1}_{A_{k} \cap \{\tau_{2n} \leq \theta_{k}\}} + \xi^{2n+1}(T)\mathbbm{1}_{\bar{A} \cap \{\tau_{2n} \leq T\}}\\ 
&= \sum_{k\in\mathbbm{N}}Y^{1}(\tau_{2n})\mathbbm{1}_{A_{k} \cap \{\tau_{2n} \leq \theta_{k}\}} + Y^{1}(\tau_{2n})\mathbbm{1}_{\bar{A} \cap \{\tau_{2n} \leq T\}} = Y^{1}(\tau_{2n})\mathbbm{1}_{\{\tau_{2n} \leq \tau\}}.
\end{align*}
iv) By the previous property (iii), we have, $V^{2n+1}(\tau) = \xi^{2n+1}(\tau)$ if and only if $V^{2n+1}(\tau)\mathbbm{1}_{\{\tau < \tau_{2n}\}} = \xi^{2n+1}(\tau)\mathbbm{1}_{\{\tau < \tau_{2n}\}}$. Hence, 
$$\tilde{\tau}_{2n+1} = \essinf\{\tau \in \Theta: V^{2n+1}(\tau) = X^{1}(\tau)\} \wedge \tau_{2n}.$$

\end{proof}
\noindent
Similarly to \eqref{TEXeQuationN2}, \eqref{TEXeQuationN22} and \eqref{TEXeQuationN55}, we define:
$$\xi^{2n+2}(\theta_{k}) \coloneqq X^{2}(\theta_{k})\mathbbm{1}_{\{\theta_{k}<\tau_{2n+1}\}} + Y^{2}(\tau_{2n+1})\mathbbm{1}_{\{\tau_{2n+1}\leq \theta_{k}\}}, \text{ and }$$
$$\xi^{2n+2}(T) \coloneqq Y^{2}(\tau_{2n+1}).$$
For $\tau \in \Theta$ of the form $\tau = \sum_{k\in\mathbbm{N}}\theta_{k}\mathbbm{1}_{A_{k}} + T\mathbbm{1}_{\bar{A}}$, we define
\begin{equation}\label{TEXeQuatioN8888}
\begin{aligned}
&\xi^{2n+2}(\tau) \coloneqq \sum_{k\in\mathbbm{N}}\xi^{2n+2}(\theta_{k})\mathbbm{1}_{A_{k}} + \xi^{2n+2}(T)\mathbbm{1}_{\bar{A}}\\
&V^{2n+2}(S) \coloneqq \esssup_{\tau \in \Theta_{S}} \rho^{2}_{S, \tau \wedge \tau_{2n+1}}[\xi^{2n+2}(\tau)]\\
&\tilde{\tau}_{2n+2} \coloneqq \essinf \tilde{\mathcal{A}}^{2}, \text{ where } \tilde{\mathcal{A}}^{2} \coloneqq \{\tau \in \Theta: V^{2n+2}(\tau) = \xi^{2n+2}(\tau)\}\\
&\tau_{2n+2} \coloneqq (\tilde{\tau}_{2n+2} \wedge \tau_{2n})\mathbbm{1}_{\{\tilde{\tau}_{2n+2} \wedge \tau_{2n} < \tau_{2n+1}\}} + \tau_{2n}\mathbbm{1}_{\{\tilde{\tau}_{2n+2} \wedge \tau_{2n} \geq \tau_{2n+1}\}}.
\end{aligned}
\end{equation}
\noindent
The random variable $\xi^{2n+2}(\tau)$ is exactly the pay-off at $\tau \wedge \tau_{2n+1}$ of agent 2, if agent 1 plays $\tau_{2n+1}$ and agent 2 plays $\tau$. Hence, $V^{2n+2}(S)$ is the optimal value at time $S$ for agent 2, when agent 1's strategy is fixed to $\tau_{2n+1}$. Assuming that $\limsup_{k \to +\infty}X^{2}(\theta_{k}) \leq X^{2}(T)$ leads to $\limsup_{k \to +\infty}\xi^{2n+2}(\theta_{k}) \leq \xi^{2n+2}(T)$, which we use in applying Theorem 2.3 in \cite{Grigorova-1}. By Theorem 2.3 in \cite{Grigorova-1}, the Bermudan stopping time $\tilde{\tau}_{2n+2}$ is optimal for the problem with value $V^{2n+2}(0)$, that is, 
$$V^{2n+2}(0) = \rho^{2}_{0, \tilde{\tau}_{2n+2} \wedge \tau_{2n+1}}[\xi^{2n+2}(\tilde{\tau}_{2n+2})] = \sup_{\tau \in \Theta}\rho^{2}_{0, \tau \wedge \tau_{2n+1}}[\xi^{2n+2}(\tau)].$$

\begin{remark}\label{TEXRemark33}
Let us recall that (cf. \cite{Grigorova-1}) $\tilde{\tau}_{n} = \essinf \{\tau \in \Theta: V^{n}(\tau) = \xi^{n}(\tau)\}$ satisfies the property: $V^{n}(\tilde{\tau}_{n}) = \xi^{n}(\tilde{\tau}_{n})$. (This is due to the property of stability of $\Theta$ by monotone limit and to the right-continuity-along Bermudan stopping strategies of the families $(V^{n}(\tau))$ and $(\xi^{n}(\tau))$).

\end{remark}

\begin{remark}\label{TEXRemark44}
By analogy with Remark \ref{TEXremark2}, we have:\\
[0.2cm] 
i) $(\xi^{2n+2}(\tau))$ is a admissible $L^{p}$-integrable family;\\
[0.2cm]
ii) for each $n \in \mathbbm{N}$, for each $\tau \in \Theta$, $\xi^{2n+2}(\tau) = \xi^{2n+2}(\tau \wedge \tau_{2n+1})$.

\end{remark}

\begin{proposition} \label{TEXProposition2}
We assume that $\rho$ satisfies the usual ``zero-one law''. Then, for all $m \geq 1$, $\tilde{\tau}_{m+2} \leq \tau_{m}$.

\end{proposition}

\begin{proof}
We suppose, by way of contradiction, that there exists $m \geq 1$ such that $P(\tilde{\tau}_{m+2} > \tau_{m}) > 0$, and we set $n \coloneqq \min\{m \geq 1: P(\tilde{\tau}_{m+2} > \tau_{m}) > 0\}$. We have $\tilde{\tau}_{n+1} \leq \tau_{n-1}$, by definition of $n$. This observation, together with the definition of $\tau_{n+1}$ and with the inequality of part (iv) of Proposition \ref{TEXProposition1} gives:
\begin{equation}\label{TEXeQuationN6}
\begin{aligned}
\tau_{n+1} &= (\tilde{\tau}_{n+1} \wedge \tau_{n-1})\mathbbm{1}_{\{\tilde{\tau}_{n+1}\wedge\tau_{n-1} < \tau_{n}\}} + \tau_{n-1}\mathbbm{1}_{\{\tilde{\tau}_{n+1}\wedge\tau_{n-1} \geq \tau_{n}\}}\\
&= \tilde{\tau}_{n+1}\mathbbm{1}_{\{\tilde{\tau}_{n+1} < \tau_{n}\}} + \tau_{n-1}\mathbbm{1}_{\{\tilde{\tau}_{n+1} \geq \tau_{n}\}}\\
&= \tilde{\tau}_{n+1}\mathbbm{1}_{\{\tilde{\tau}_{n+1} < \tau_{n}\}} + \tau_{n-1}\mathbbm{1}_{\{\tilde{\tau}_{n+1} = \tau_{n}\}}
\end{aligned}
\end{equation}
For similar reasons, we have 
\begin{equation}\label{TEXeQuationN7}
\tau_{n} = \tilde{\tau}_{n}\mathbbm{1}_{\{\tilde{\tau}_{n} < \tau_{n-1}\}} + \tau_{n-2}\mathbbm{1}_{\{\tilde{\tau}_{n}=\tau_{n-1}\}}.
\end{equation}
For the easing of the presentation, we set $\Gamma \coloneqq \{\tau_{n}<\tilde{\tau}_{n+2}\}$.\\
[0.2cm]
On the set $\Gamma$, we have:\\
[0.2cm]
1) $\tau_{n} < \tilde{\tau}_{n+2} \leq \tau_{n+1}$, the last inequality being due to property (iv) of Proposition \ref{TEXProposition1}.\\
[0.2cm]
2) $\tau_{n+1} = \tau_{n-1}$. This is due to (1), together with Eq. \eqref{TEXeQuationN6}.\\
[0.2cm]
3) $\xi^{n+2} = \xi^{n}$. This is a consequence of (2) and the definitions of $\xi^{n+2}$ and $\xi^{n}$.\\
[0.2cm]
4) $\tau_{n} = \tilde{\tau}_{n}$.\\
[0.2cm] 
We prove that $\{\tilde{\tau}_{n} = \tau_{n-1}\}\cap\Gamma=\emptyset$, which together with Eq. \eqref{TEXeQuationN7}, gives the desired statement.\\
[0.2cm]
Due to Eq. \eqref{TEXeQuationN7}, we have $\{\tilde{\tau}_{n} = \tau_{n-1}\} = \{\tilde{\tau}_{n} = \tau_{n-1}\}\cap\{\tau_{n} = \tau_{n-2}\}$. Thus, we have 
$$\{\tilde{\tau}_{n} = \tau_{n-1}\} \cap \Gamma = \{\tilde{\tau}_{n} = \tau_{n-1}, \tau_{n} = \tau_{n-2}<\tilde{\tau}_{n+2}\}.$$
Now, we have $\tilde{\tau}_{n} \leq \tau_{n-2}$ (due to the definition of $n$). Hence, 
$$\{\tilde{\tau}_{n} = \tau_{n-1}\}\cap\Gamma = \{\tilde{\tau}_{n} = \tau_{n-1} \leq \tau_{n-2} = \tau_{n} < \tilde{\tau}_{n+2}\} = \emptyset,$$
where the equality with $\emptyset$ is due to $\tilde{\tau}_{n+2} \leq \tau_{n-1}$.\\
[0.2cm]
We note that combining properties (1) and (4) gives $\tilde{\tau}_{n} < \tilde{\tau}_{n+2}$ on $\Gamma$. We will obtain a contradiction with this property. To this end, we will show that:
\begin{equation}\label{TEXeQuationN8}
\mathbbm{1}_{\Gamma}V^{n+2}(\tilde{\tau}_{n}) = \mathbbm{1}_{\Gamma}\xi^{n+2}(\tilde{\tau}_{n}).
\end{equation}
By definition of $\tilde{\tau}_{n}$ and by Remark \ref{TEXRemark33}, we have:
$$V^{n}(\tilde{\tau}_{n}) = \xi^{n}(\tilde{\tau}_{n}).$$
This property, together with property (3) on $\Gamma$, gives $V^{n}(\tilde{\tau}_{n}) = \xi^{n}(\tilde{\tau}_{n}) = \xi^{n+2}(\tilde{\tau}_{n})$ on $\Gamma$. In order to show Eq. \eqref{TEXeQuationN8}, it suffices to show 
$$\mathbbm{1}_{\Gamma}V^{n+2}(\tilde{\tau}_{n}) = \mathbbm{1}_{\Gamma}V^{n}(\tilde{\tau}_{n}).$$
By property (4) on $\Gamma$ and Proposition \ref{TEXproposition-localisation} (applied with $A=\Gamma \in \mathcal{F}_{\tau_{n}}$ and $\tau = \tau_{n}$), we have
$$\mathbbm{1}_{\Gamma}V^{n+2}(\tilde{\tau}_{n}) = \mathbbm{1}_{\Gamma}V^{n+2}(\tau_{n}) = \mathbbm{1}_{\Gamma}V_{\Gamma}^{n+2}(\tau_{n}).$$
Due to property (3) on $\Gamma$, $V^{n+2}_{\Gamma}$ and $V^{n}_{\Gamma}$ have the same pay-off, and by applying again Proposition \ref{TEXproposition-localisation} and property (4) on $\Gamma$, we have
$$\mathbbm{1}_{\Gamma}V_{\Gamma}^{n+2}(\tau_{n}) = \mathbbm{1}_{\Gamma}V_{\Gamma}^{n}(\tau_{n}) = \mathbbm{1}_{\Gamma}V^{n}(\tau_{n}) = \mathbbm{1}_{\Gamma}V^{n}(\tilde{\tau}_{n}).$$
We have shown that $\mathbbm{1}_{\Gamma}V^{n+2}(\tilde{\tau}_{n}) = \mathbbm{1}_{\Gamma}V^{n}(\tilde{\tau}_{n})$, which is the desired equality. Hence, we get $\tilde{\tau}_{n+2} \leq \tilde{\tau}_{n}$ on $\Gamma$ (as by definition $\tilde{\tau}_{n+2} = \essinf\{\tau \in \Theta: V^{n+2}(\tau) = \xi^{n+2}(\tau)\}$). However, this is in contradiction with the property $\tilde{\tau}_{n+2} > \tilde{\tau}_{n}$ on $\Gamma$. The proof is complete.

\end{proof}

\begin{lemma}\label{TEXlemma1111}
i) For all $n \geq 2$, $\tau_{n+1} = \tilde{\tau}_{n+1}\mathbbm{1}_{\{\tilde{\tau}_{n+1} < \tau_{n}\}} + \tau_{n-1}\mathbbm{1}_{\{\tilde{\tau}_{n+1} = \tau_{n}\}}$.\\
[0.2cm]
ii) For all $n \geq 2$, $\tilde{\tau}_{n+1} = \tau_{n+1} \wedge \tau_{n}$.\\
[0.2cm]
iii) On $\{\tau_{n} = \tau_{n-1}\}$, $\tau_{m} = T$, for all $m \in \{1, ..., n\}$.

\end{lemma}

\begin{proof}
i) This property follows from the definition of $\tau_{n+1}$, together with Proposition \ref{TEXProposition2}, and with property (iv) of Proposition \ref{TEXProposition1}.\\
[0.3cm]
ii) By using (i), we get 
\begin{align*}
\tau_{n+1}\wedge\tau_{n} &= (\tilde{\tau}_{n+1}\wedge\tau_{n})\mathbbm{1}_{\{\tilde{\tau}_{n+1}<\tau_{n}\}} + (\tau_{n-1}\wedge\tau_{n})\mathbbm{1}_{\{\tilde{\tau}_{n+1}=\tau_{n}\}}\\
&= \tilde{\tau}_{n+1}\mathbbm{1}_{\{\tilde{\tau}_{n+1}<\tau_{n}\}} + (\tau_{n-1}\wedge\tilde{\tau}_{n+1})\mathbbm{1}_{\{\tilde{\tau}_{n+1}=\tau_{n}\}}\\
&= \tilde{\tau}_{n+1}\mathbbm{1}_{\{\tilde{\tau}_{n+1}<\tau_{n}\}} + \tilde{\tau}_{n+1}\mathbbm{1}_{\{\tilde{\tau}_{n+1}=\tau_{n}\}},
\end{align*}
where we have used Proposition \ref{TEXProposition2} for the last equality.\\
[0.2cm]
Finally, by using property (iv) of Proposition \ref{TEXProposition1}, we get
$\tau_{n+1} \wedge \tau_{n} = \tilde{\tau}_{n+1}.$\\
[0.3cm]
iii) To prove this property, we proceed by induction. The property is true for $n=2$. We suppose that the property is true at rank $n-1$ (where $n \geq 3$), that is on $\{\tau_{n-1} = \tau_{n-2}\}$, $\tau_{m}=T$, for all $m \in \{1, ..., n-1\}$.\\
[0.2cm]
From the expression for $\tau_{n}$ from statement (i), we get 
$$\tau_{n} = \tilde{\tau}_{n}\mathbbm{1}_{\{\tilde{\tau}_{n}<\tau_{n-1}\}} + \tau_{n-2}\mathbbm{1}_{\{\tilde{\tau}_{n}=\tau_{n-1}\}}.$$
Hence, $\tau_{n} = \tau_{n-2}$ on the set $\{\tau_{n} = \tau_{n-1}\}$. We conclude by the induction hypothesis.
\end{proof}

\begin{lemma}\label{TEXlemma2222}
The following inequalities hold true:\\
[0.2cm]
i) $J_{1}(\tau, \tau_{2n}) \leq J_{1}(\tau_{2n+1}, \tau_{2n}), \text{ for all } \tau \in \Theta$.\\
[0.2cm]
ii) $J_{2}(\tau_{2n+1}, \tau) \leq J_{2}(\tau_{2n+1}, \tau_{2n+2}), \text{ for all } \tau \in \Theta$.

\end{lemma}

\begin{proof}
Let us first prove statement i):\\
[0.2cm]
By (A1), we have $X^{1} \leq Y^{1}$; it follows 
$$X^{1}(\tau)\mathbbm{1}_{\{\tau \leq \tau_{2n}\}} + Y^{1}(\tau_{2n})\mathbbm{1}_{\{\tau_{2n} < \tau\}} \leq \xi^{2n+1}(\tau).$$
Hence, by monotonicity, and by definition of $V^{2n+1}(0)$, we have
\begin{equation}\label{TEXeQuationN1212}
J_{1}(\tau, \tau_{2n}) \leq V^{2n+1}(0)
\end{equation}
We will now show that $V^{2n+1}(0) = J_{1}(\tau_{2n+1}, \tau_{2n})$, which will complete the proof of statement i).\\
[0.2cm]
We have 
\begin{align*}
J_{1}(\tau_{2n+1}, \tau_{2n}) &= \rho^{1}_{0, \tau_{2n+1} \wedge \tau_{2n}}(X^{1}(\tau_{2n+1})\mathbbm{1}_{\{\tau_{2n+1} \leq \tau_{2n}\}} + Y^{1}(\tau_{2n})\mathbbm{1}_{\{\tau_{2n} < \tau_{2n+1}\}})\\
&= \rho^{1}_{0, \tau_{2n+1} \wedge \tau_{2n}}(\xi^{2n+1}(\tau_{2n+1})),
\end{align*}
where we have used iii) from Lemma \ref{TEXlemma1111}, and $X^{1}(T) = Y^{1}(T)$ from (A2) to show the last equality.\\
[0.2cm]
On the other hand, by ii) from Remark \ref{TEXremark2} and ii) from Lemma \ref{TEXlemma1111},
\begin{align*}
\rho^{1}_{0, \tau_{2n+1} \wedge \tau_{2n}}(\xi^{2n+1}(\tau_{2n+1})) = \rho^{1}_{0, \tau_{2n+1} \wedge \tau_{2n}}(\xi^{2n+1}(\tau_{2n+1} \wedge \tau_{2n})) = \rho^{1}_{0, \tilde{\tau}_{2n+1}\wedge \tau_{2n}}(\xi^{2n+1}(\tilde{\tau}_{2n+1})).
\end{align*}
By optimality of $\tilde{\tau}_{2n+1}$ for $V^{2n+1}(0)$ (cf. Eq. \eqref{TEXeQuationN6666}), we get
$$\rho^{1}_{0, \tilde{\tau}_{2n+1}\wedge \tau_{2n}}(\xi^{2n+1}(\tilde{\tau}_{2n+1})) = V^{2n+1}(0).$$
Hence, we conclude
\begin{equation}\label{TEXequatioN1313}
J_{1}(\tau_{2n+1}, \tau_{2n}) = \rho^{1}_{0, \tau_{2n+1} \wedge \tau_{2n}}(\xi^{2n+1}(\tau_{2n+1}))
= \rho^{1}_{0, \tilde{\tau}_{2n+1}\wedge \tau_{2n}}(\xi^{2n+1}(\tilde{\tau}_{2n+1})) = V^{2n+1}(0).
\end{equation}
From Eq. \eqref{TEXequatioN1313} and Eq. \eqref{TEXeQuationN1212}, we get
$$J_{1}(\tau, \tau_{2n}) \leq J_{1}(\tau_{2n+1}, \tau_{2n}).$$ 
Let us now prove statement ii):\\
[0.2cm]
We have 
\begin{equation}\label{TEXeQuaTioN1414}
J_{2}(\tau_{2n+1}, \tau) \leq V^{2n+2}(0),
\end{equation}
by definition of $V^{2n+2}(0)$ (cf. Eq. \eqref{TEXeQuatioN8888}).\\
[0.2cm]
We will now show that $J_{2}(\tau_{2n+1}, \tau_{2n+2}) = V^{2n+2}(0)$, which will complete the proof.\\
[0.2cm]
By definition of $\xi^{2n+2}(\tau_{2n})$, by ii) from Remark \ref{TEXRemark44}, and by ii) from Lemma \ref{TEXlemma1111}, we have 
\begin{align*}
J_{2}(\tau_{2n+1}, \tau_{2n+2}) &= \rho^{2}_{0, \tau_{2n+1} \wedge \tau_{2n+2}}(X^{2}(\tau_{2n+2})\mathbbm{1}_{\{\tau_{2n+2} < \tau_{2n+1}\}} + Y^{2}(\tau_{2n+1})\mathbbm{1}_{\{\tau_{2n+1} \leq \tau_{2n+2}\}})\\
&= \rho^{2}_{0, \tau_{2n+1} \wedge \tau_{2n+2}}(\xi^{2n+2}(\tau_{2n+2}))\\ 
&= \rho^{2}_{0, \tau_{2n+1} \wedge \tau_{2n+2}}(\xi^{2n+2}(\tau_{2n+2} \wedge \tau_{2n+1}))\\
&= \rho^{2}_{0, \tilde{\tau}_{2n+2} \wedge \tau_{2n+1}}(\xi^{2n+2}(\tilde{\tau}_{2n+2})).
\end{align*}
By Eq. \eqref{TEXeQuatioN8888}, we have 
$$\rho^{2}_{0, \tilde{\tau}_{2n+2} \wedge \tau_{2n+1}}(\xi^{2n+2}(\tilde{\tau}_{2n+2})) = V^{2n+2}(0).$$
Hence, we conclude
\begin{equation}\label{TEXequatioN1515}
J_{2}(\tau_{2n+1}, \tau_{2n+2}) = \rho^{2}_{0, \tau_{2n+1} \wedge \tau_{2n+2}}(\xi^{2n+2}(\tau_{2n+2}))
= \rho^{2}_{0, \tilde{\tau}_{2n+2}\wedge \tau_{2n+1}}(\xi^{2n+2}(\tilde{\tau}_{2n+2})) = V^{2n+2}(0).
\end{equation}
From Eq. \eqref{TEXequatioN1515} and Eq. \eqref{TEXeQuaTioN1414}, we get 
$$J_{2}(\tau_{2n+1}, \tau) \leq J_{2}(\tau_{2n+1}, \tau_{2n+2}).$$
\end{proof}

\begin{remark}
As a by-product of the previous proof, we find that $\tau_{2n+1}$ is optimal for the problem with value $V^{2n+1}(0)$, and $\tau_{2n+2}$ is optimal for the problem with value $V^{2n+2}(0)$.

\end{remark}

\begin{definition}
We define $\tau_{1}^{*} = \lim_{n \to +\infty}\tau_{2n+1}$, and $\tau_{2}^{*} = \lim_{n \to +\infty}\tau_{2n}$.

\end{definition}

\begin{proposition}\label{TEXproposition333}
We assume that $\rho^{1}$ and $\rho^{2}$ satisfy properties $(i) - (vii)$, and the following additional property: for $i \in \{1, 2\}$, 
\begin{equation}\label{TEXEquatiON1616}
\limsup_{n \to +\infty}\rho^{i}_{0, \nu_{n}}[\xi(\nu_{n})] = \rho^{i}_{0, \nu}[\xi(\nu)],
\end{equation}
for any sequence $(\nu_{n}) \subset \Theta^{\mathbbm{N}}$, $\nu \in \Theta$, such that $\nu_{n} \downarrow \nu$. We have:\\
[0.2cm]
i) For all $\tau \in \Theta$, $\lim_{n \to +\infty}J_{1}(\tau, \tau_{2n}) = J_{1}(\tau, \tau_{2}^{*})$.\\
[0.2cm]
ii) For all $\tau \in \Theta$, $\lim_{n \to +\infty}J_{2}(\tau_{2n+1}, \tau) = J_{2}(\tau_{1}^{*}, \tau)$.\\
[0.2cm]
iii) For all $\tau \in \Theta$, $\lim_{n \to +\infty}J_{1}(\tau_{2n+1}, \tau_{2n+2}) = J_{1}(\tau_{1}^{*}, \tau_{2}^{*})$.\\
[0.2cm]
iv) For all $\tau \in \Theta$, $\lim_{n \to +\infty}J_{2}(\tau_{2n+1}, \tau_{2n+2}) = J_{2}(\tau_{1}^{*}, \tau_{2}^{*})$.

\end{proposition}

\begin{proof}
Let us first show statement i):\\
[0.2cm]
Let us recall the following notation:\\
[0.2cm]
for a fixed $\tau \in \Theta$,
$$I^{1}(\tau, \nu) \coloneqq X^{1}(\tau)\mathbbm{1}_{\{\tau \leq \nu\}} + Y^{1}(\nu)\mathbbm{1}_{\{\nu < \tau\}},$$
$$I^{1}(\tau, \tau_{2}^{*}) \coloneqq X^{1}(\tau)\mathbbm{1}_{\{\tau \leq \tau_{2}^{*}\}} + Y^{1}(\tau_{2}^{*})\mathbbm{1}_{\{\tau_{2}^{*} < \tau\}}.$$
With this notation, we have 
$$J_{1}(\tau, \tau_{2n}) = \rho^{1}_{0, \tau \wedge \tau_{2n}}[I(\tau, \tau_{2n})], \text{ \;\; and \;\; } J_{1}(\tau, \tau_{2}^{*}) = \rho^{1}_{0, \tau \wedge \tau_{2}^{*}}[I(\tau, \tau_{2}^{*})].$$
We note that the sequence $(\tau_{2n})$ and $(\tau \wedge \tau_{2n})$ converges from above to $\tau_{2}^{*}$ and $\tau \wedge \tau_{2}^{*}$, respectively. Moreover, for each $\tau \in \Theta$, the family $(I^{1}(\tau, \nu))_{\nu \in \Theta}$ is admissible. Indeed, for each $\nu \in \Theta$, $I^{1}(\tau, \nu)$ is $\mathcal{F}_{\nu}$-measurable. Moreover, if $\{\nu = \nu'\}$, $I(\tau, \nu) = I(\tau, \nu')$ a.s. Hence, as any admissible family in our framework is right-continuous along Bermudan stopping strategies (cf. Remark 2.10 in \cite{Grigorova-1}), we get
$$\lim_{n \to +\infty}I(\tau, \tau_{2n}) = I(\tau, \tau_{2}^{*}).$$
Hence, by property \eqref{TEXEquatiON1616} on $\rho^{1}$, we get
$$\limsup_{n \to +\infty}\rho^{1}_{0, \tau \wedge \tau_{2n}}[I(\tau, \tau_{2n})] = \rho^{1}_{0, \tau \wedge \tau_{2}^{*}}[I(\tau, \tau_{2}^{*})].$$
Now, let us prove statement ii):\\
[0.2cm]
For $\tau \in \Theta$, we recall the following notation:
$$I^{2}(\nu, \tau) \coloneqq X^{2}(\tau)\mathbbm{1}_{\{\tau < \nu\}} + Y^{2}(\nu)\mathbbm{1}_{\{\nu \leq \tau\}},$$
The family $(I^{2}(\nu, \tau)_{\nu \in \Theta}$ is admissible. Indeed, for each $\nu \in \Theta$, $I^{2}(\nu, \tau)$ is $\mathcal{F}_{\nu}$-measurable. Moreover, on $\{\nu_{1} = \nu_{2}\}$, $I^{2}(\nu_{1}, \tau) = I^{2}(\nu_{2}, \tau)$ a.s.\\
[0.2cm]
As $(\tau_{2n+1})$ converges from above to $\tau_{1}^{*}$, and as $(I^{2}(\nu, \tau)_{\nu \in \Theta})$ is right-continuous along Bermudan stopping strategies (cf. Remark 2.10 in \cite{Grigorova-1}), we get
$$\lim_{n \to +\infty}I^{2}(\tau_{2n+1}, \tau) = I^{2}(\tau_{1}^{*}, \tau).$$
By property \eqref{TEXEquatiON1616} on $\rho^{2}$, we get
$$\limsup_{n \to +\infty}\rho^{2}_{0, \tau_{2n+1} \wedge \tau}[I^{2}(\tau_{2n+1}, \tau)] = \rho^{2}_{0, \tau_{1}^{*} \wedge \tau}[I^{2}(\tau_{1}^{*}, \tau)].$$
We now prove statement iii).\\
[0.2cm]
The proof relies again on the Bermudan structure on $\Theta$. For any sequence $(\tau_{n}) \in \Theta^{\mathbbm{N}}$ converging from above to $\tau \in \Theta$, we have: for almost each $\omega \in \Omega$, there exists $n_{0} = n_{0}(\omega)$ such that for all $n \geq n_{0}$, $\tau_{n}(\omega) = \tau(\omega)$ (cf. Remark 10 in \cite{Grigorova-1}).\\
[0.2cm]
Hence, for almost each $\omega \in \Omega$, there exists $n_{0} = n_{0}(\omega)$ such that for $n \geq n_{0}$, 
$\tau_{2n+1}(\omega) = \tau_{1}^{*}(\omega)$, $\tau_{2n+2}(\omega) = \tau_{2}^{*}(\omega)$ and 
$$I(\tau_{2n+1}, \tau_{2n+2})(\omega) = X^{1}(\tau_{1}^{*})(\omega)\mathbbm{1}_{\{\tau_{1}^{*} \leq \tau_{2}^{*}\}}(\omega) + Y^{1}(\tau_{2}^{*})(\omega)\mathbbm{1}_{\{\tau_{2}^{*} < \tau_{1}^{*}\}}(\omega).$$
By property \eqref{TEXEquatiON1616} on $\rho^{1}$, we get
$$\lim_{n \to +\infty}J_{1}(\tau_{2n+1}, \tau_{2n+2}) = J_{1}(\tau_{1}^{*}, \tau_{2}^{*}).$$
The proof of iv) is based on the same arguments.
\end{proof}
\noindent We are now ready to complete the proof of Theorem \ref{TEXtheorem1}.
\begin{proof}[Proof of Theorem \ref{TEXtheorem1}]
By combining Lemma \ref{TEXlemma2222} and Proposition \ref{TEXproposition333}, we get:
$$J_{1}(\tau,\tau_{2}^{*}) \leq J_{1}(\tau_{1}^{*}, \tau_{2}^{*}), \text{ for all } \tau \in \Theta.$$
$$J_{2}(\tau_{1}^{*},\tau) \leq J_{2}(\tau_{1}^{*}, \tau_{2}^{*}), \text{ for all } \tau \in \Theta.$$
Hence, $(\tau_{1}^{*}, \tau_{2}^{*})$ is a Nash equilibrium point.
\end{proof}
 We have thus shown that the non-linear non-zero-sum Dynkin game with Bermudan strategies admits a Nash equilibrium point. 
\newpage
\section{Appendix}
\begin{proposition}\label{TEXproposition-localisation}(Localisation property) 
Let $(\xi(\tau))_{\tau \in \Theta}$ be a given admissible $p$-integrable family. Let $(V(\tau))_{\tau \in \Theta}$ be the value family of the optimal stopping problem: for $S \in \Theta$, 
$$V(S) = \esssup_{\tau \in \Theta_{S}}\rho_{S, \tau}[\xi(\tau)].$$
Let $S \in \Theta$, and let $A$ be in $\mathcal{F}_{S}$. We consider the pay-off family $(\xi(\tau)\mathbbm{1}_{A})_{\tau \in \Theta_{S}}$, and we denote by $(V_{A}(\tau))_{\tau \in \Theta_{S}}$ the corresponding value family, defined by:
$$V_{A}(\tau) = \esssup_{\nu \in \Theta_{\tau}}\rho_{\tau, \nu}[\xi(\nu)\mathbbm{1}_{A}].$$
If $\rho$ satisfies the usual ``zero-one law'' (that is $\mathbbm{1}_{A}\rho_{S, \tau}[\eta] = \mathbbm{1}_{A}\rho_{S, \tau}[\mathbbm{1}_{A}\eta]$ for all $A \in \mathcal{F}_{S}$, for all $\eta \in L^{p}(\mathcal{F}_{\tau})$), then for each $\tau \in \Theta_{S}$,
$$\mathbbm{1}_{A}V_{A}(\tau) = \mathbbm{1}_{A}V(\tau).$$

\end{proposition}

\begin{proof}
By the definition of $V(\tau)$ and the usual ``zero-one law'', we have
\begin{align*}
\mathbbm{1}_{A}V(\tau) &= \mathbbm{1}_{A}\esssup_{\nu \in \Theta_{\tau}}\rho_{\tau, \nu}[\xi(\nu)] = \esssup_{\nu \in \Theta_{\tau}}\mathbbm{1}_{A}\rho_{\tau, \nu}[\xi(\nu)]\\
&= \esssup_{\nu \in \Theta_{\tau}}\mathbbm{1}_{A}\rho_{\tau, \nu}[\mathbbm{1}_{A}\xi(\nu)] = \mathbbm{1}_{A}\esssup_{\nu \in \Theta_{\tau}}\rho_{\tau, \nu}[\mathbbm{1}_{A}\xi(\nu)] = \mathbbm{1}_{A}V_{A}(\tau).
\end{align*}

\end{proof}

\newpage
\bibliographystyle{plainnat}

\end{document}